\documentclass[11pt,a4paper]{article}

\usepackage{amsmath}
\usepackage{graphicx}
\usepackage{a4}                          
\usepackage[english]{babel}               
\usepackage[latin1]{inputenc}            
\usepackage{amssymb}                     
\usepackage{graphicx}                    
\usepackage{verbatim}
\usepackage{rotating}
\usepackage{booktabs}
\usepackage{amsthm}

\title{Asymptotics for the ruin time of a piecewise exponential Markov process with jumps}
\author{Anders R\o nn-Nielsen\\
\normalsize Department of Mathematical Sciences, 
University of Copenhagen}
\date{}

\newcommand{\dd}{{\mathrm{d}}}

\newcommand{\ee}{{\mathrm{e}}}

\newcommand{\RR}{\mathbb {R}}
\newcommand{\CC}{\mathbb {C}}

\newcommand{\PPx}{\mathbb {P}^x}
\newcommand{\EE}{\mathbb {E}}
\newcommand{\EEx}{\mathbb {E}^x}
\newcommand{\ZZ}{\mathbb {Z}}

\newcommand{\AAA}{\mathcal {A}}

\newtheorem{Th}{Theorem}[section]
\newtheorem{Lemma}{Lemma}[section]
\newtheorem{Ex}{Example}[section]

\newtheorem{Cor}{Corollary}[section]
\newtheorem{Def}{Definition}[section]
\newtheorem{Con}{Condition}[section]
\newtheorem{Rem}{Remark}[section]

\begin{document}\maketitle

\begin{abstract}
In this paper a class of Ornstein--Uhlenbeck processes driven by
compound Poisson processes is considered. The jumps arrive with exponential waiting times and are
allowed to be two-sided. The jumps are assumed to form an iid sequence
with distribution a mixture (not necessarily convex) of exponential distributions, independent of
everything else. The fact that downward jumps are allowed makes
passage of a given lower level possible both by continuity and by a
jump. The time of this passage and the possible undershoot (in the
jump case) is considered. By finding partial eigenfunctions for the
infinitesimal generator of the process, an expression for the joint
Laplace transform of the passage time and the undershoot can be found.

From the Laplace transform the ruin probability of ever crossing the
level can be derived. When the drift is negative this probability is
less than one and its asymptotic behaviour when the initial state of the process
tends to infinity is determined explicitly.

The situation where the level to cross decreases to minus infinity is
more involved: The level to cross plays a much more fundamental
role in the expression for the joint Laplace transform than the
initial state of the process. The limit of the ruin probability in the
positive drift case and the limit of the distribution of the
undershoot in the negative drift case is derived.

{\small \noindent\emph{Keywords:} Asymptotic ruin probabilities; Integration contour; Ornstein--Uhlenbeck process; Partial Eigenfunction; Shot--noise process}

\end{abstract}


\section{Introduction}
The main aim of this paper is to determine the asymptotic behaviour of
the ruin probability for a certain class of time--homogeneous Markov
processes with jumps. These processes, referred to as $X$ below, can be viewed as
Ornstein--Uhlenbeck processes satisfying
\begin{equation}\label{ligning11}
\dd X_t=\kappa X_t\, \dd t+\dd U_t\,,
\end{equation}
driven by a compound Poisson process $(U_t)$. The ruin time, $\tau(\ell)$, is defined as the time to passage below $\ell$ for an
initial state $x>\ell$. The passage below $\ell$ can be a result of a
downward jump, and in some cases a continuous passage through $\ell$ is
is also possible. The main results give asymptotic descriptions of $\PPx(\tau(\ell)<\infty)$, when $\kappa>0$ in the limits $x\to\infty$ and $\ell\to-\infty$. Furthermore, the limit distribution of the undershoot in case of passage by jump is determined for $\kappa<0$ and $\ell\to-\infty$.

It will be assumed that the driving compound Poisson process has
a special jump structure. Both the downward and upward jumps are
assumed to have a density (not the same) that is a linear -- not
necessarily convex -- combination of exponential densities

It is important to distinguish between two different scenarios:
Whether the drift $\kappa$ is positive, hence $X$ is transient, or the drift is negative, in which case the process $X$
is recurrent. In the negative drift case the probability
$\PPx(\tau(\ell)<\infty)$ (with $\tau(\ell)$ denoting the time of
passage) of ever crossing below $\ell$ when starting
at $x$ is always 1. When the drift $\kappa$ is
positive we have that $\PPx(\tau(\ell)<\infty)<1$, and this
probability decreases when either $x\to\infty$ or $\ell\to-\infty$.

The distribution of the passage time (and by that also the ruin
probability) is determined through the Laplace transform. This is
found by exploiting certain stopped martingales derived from using
bounded partial eigenfunctions for the infinitesimal generator for
$X$. An explicit expression for the Laplace transform is determined in
\cite{art3}. Here the partial eigenfunctions are found as linear
combinations of functions given by contour integrals in the complex
plane. Also the Laplace transform ends up being a linear combination
of these integrals. It is the resulting Laplace transform from
\cite{art3} that we shall investigate in this paper.

In the present paper the asymptotics of $\PPx(\tau(\ell)<\infty)$
is explored in both of the situations $x\to\infty$ and
$\ell\to-\infty$. This becomes a question about finding the asymptotics
for the complex contour integrals mentioned above. It turns out that
the $\ell\to-\infty$ problem is the far most complicated because the dependence of $\ell$ in the
construction of the partial eigenfunctions is more
involved. Nevertheless, the need of exploring the asymptotic behaviour
of the integrals is similar. When $x\to\infty$ we see that
$\PPx(\tau(\ell)<\infty)$ decreases exponentially (adjusted by some
specified power function) with the exponential parameter from the
leading exponential part of the downward jumps. 

The technique of using partial eigenfunctions for the infinitesimal
generator has appeared before. Paulsen and Gjessing, \cite{gjessing},
considers a model like the present, but in the more general (and
also different) setup
\begin{equation}\label{stormodel}
dX_t=(p+\kappa X_t)\, dt-dU_t+\sqrt{\sigma_1^2+\sigma_2^2X_t^2}\,dB_t\,+X_td\tilde{U}_t\,.
\end{equation}
Here both $U$ and $\tilde{U}$ are compound Poisson processes of the
form $\sum_{n=1}^{N_t}V_n$. In
\cite{gjessing} it is shown that a partial eigenfunction for the
corresponding infinitesimal generator for (\ref{stormodel}) will lead
to the ruin probability  and also the Laplace transform for the ruin
time. \cite{gaier} shows in a model without $\sigma_1^2$
and $\tilde{U}$ the existence of this partial
eigenfunction under some smoothness assumptions about the jump distributions in $U$. This
result is extended to weaker assumptions in \cite{grandits}.

In the case of $\sigma_1^2=\sigma_2^2=0$, without $\tilde{U}$, and assuming exponential negative jump and no positive
jumps, an explicit formula for the Laplace transform is
determined in \cite{gjessing}. Furthermore, the exponential decrease in
$\PPx(\tau(\ell)<\infty)$ is derived in the $x\to\infty$ asymptotic
situation for some fixed $0<\ell<x$. For the case of exponential negative jumps also see
Asmussen \cite{asm}, Chapter VII.

In the present paper the jump distributions are assumed to be light
tailed. The existing literature does not contain very explicit results
for the asymptotic ruin probability with that kind of jump
distributions. In \cite{embrechts} and \cite{schmidli} it is proved in the $\sigma_2^2=0$
case with $\kappa=\sup\{a\;|\;\EE[e^{aU}]<\infty\}$ that for any $\epsilon>0$
$$
\lim_{x\to\infty}\ee^{(\kappa-\epsilon)x}\PPx(\tau_{\ell}<\infty)=0\quad\textrm{and}\quad\lim_{x\to\infty}\ee^{(\kappa+\epsilon)x}\PPx(\tau_{\ell}<\infty)=\infty\,.
$$
In the case of heavy tailed jump distributions there are more explicit
results for the asymptotic behaviour of the ruin probability. In
\cite{jiang} results are obtained for the asymptotics of the finite
horizon ruin probability $\PPx(\tau(\ell)\leq T)$ in a fairly general
model with $\sigma_2^2=0$ and subexpontial jump distributions. Similar
results are reached in \cite{chen} in the infinite horizon case. Here
the jumps belong to a less general class of heavy tailed distributions.

In \cite{art1}, \cite{art2}, \cite{art4} the following model class of certain
Markov modulated L\'evy processes 
$$
X_t=x+\int_0^t\beta_{J_s}\,\dd s +\int_0^t\sigma_{J_{s-}}\dd
B_s-\sum_{n=1}^{N_t}U_n
$$
is studied. The same partial eigenfunction technique is applied, and it
is showed that the partial eigenfunctions (and thereby also the ruin
probabilities) can be expressed as
a linear combination of exponential functions (evaluated in the starting point
$x$). Hence, the asymptotic behaviour of the probability when
$x\to\infty$ reduces to finding the exponential function
with the slowest decrease. Since the model is additive, the level $\ell$ that
is to be crossed at the time of ruin, enters into the setup
symmetric to $x$. Hence, the asymptotics when $\ell\to-\infty$ are just
as easy to derive.
In Novikov et. al, \cite{novikov1}, the Laplace transform is
determined for a shot--noise model with exponentially distributed
downward jumps (and no positive jumps allowed) for a process with
negative drift. The Laplace transform
was also derived in the case of uniformly distributed downward jumps.
In \cite{novikov2} these results are extended to a more general
driving L\'evy process instead of a compound Poisson process. In \cite{novikov2,novikov1} some asymptotic results for the distribution of $\tau(\ell)$ are carried out. Here the limit
distribution of $\tau(\ell)$ is expressed when $\ell\to-\infty$ for
some fixed starting point $x$ and negative drift. This is a limit that is not considered in the present paper.

The paper is organised as follows. In Section
\ref{setup} the setup is
defined and the relevant results from \cite{art3} reproduced. Theorem~\ref{satningfin} is also reformulated in a different (and
appearently more complicated) version as Theorem~\ref{storogfin} that turns out to fit the
asymptotic considerations better. In Section~\ref{contours} the choice of some complex integration contours
that are applied in Theorem~\ref{satningfin} and Theorem~\ref{storogfin}
is discussed. This choice differs from the proposed contours in
\cite{art3} in order to suit the further calculations. In Section
\ref{xasympt} the asymptotic behaviour of $\PPx(\tau(\ell)<\infty)$ is
expressed when $x\to\infty$ and in Section \ref{lasympt} the limit
when $\ell\to-\infty$ is found. Finally the limit of the distribution of the
undershoot is expressed for the negative drift case when $\ell\to-\infty$.

\section{The model and previous results}\label{setup}
Consider a process $X$ with state space $\RR$ defined by the following stochastic differential
equation:
\begin{equation}\label{ligning1}
\dd X_t=\kappa X_t\, \dd t+\dd U_t\,,
\end{equation}
where $(U_t)$ is a compound
Poisson process defined by
\begin{equation}\label{compound}
U_t=\sum_{n=1}^{N_t}V_n.
\end{equation}
Here $(V_n)$ are iid with distribution $G$ and $(N_t)$ is a Poisson
process with parameter $\lambda$. Both the downward
and the upward part of the jump distribution $G$ is assumed to be  a linear combination of exponential
distributions. We use the decomposition $G=pG_-+qG_+$ where $0<p\leq 1$,
$q=1-p$, $G_-$ is restricted to $\RR_-=(-\infty;0)$ and $G_+$ is
restricted to $\RR_+=(0;\infty)$. That is,
\begin{align}\label{fml:gdensity}
G_-(du)&=g_-(u)\,du=\sum_{k=1}^r\alpha_k\mu_k\ee^{\mu_ku}\qquad
\textrm{for } u<0\nonumber\\
G_+(du)&=g_+(u)\,du=\sum_{d=1}^s\beta_d\nu_d\ee^{-\nu_du}\qquad
\textrm{for } u>0\,. 
\end{align} 
The distribution parameters are arranged such that
$0<\mu_1<\cdots<\mu_r$, $0<\nu_1<\cdots<\nu_s$ and
$\alpha_i,\beta_j\neq 0$. Since $g_-$ and
$g_+$ need to be densities $\sum\alpha_i=1$ and
$\sum\beta_j=1$. Furthermore both $\alpha_1>0$ and $\beta_1>0$. The remaining density parameters are not necessarily non--negative.\\[2mm]
Between jumps the solution process $X$ behaves deterministically following an
exponential function. Assume $x>0$ and write $\PPx$ for the probability space, where
$X_0=x$ $\PPx$--almost surely. Let $\EEx$ be the corresponding
expectation. Define for $\ell<x$ the stopping time $\tau$ by
\begin{equation}\label{tau}
\tau = \tau(\ell)=\inf\{t>0\;:\;X_t\leq \ell\}\,.
\end{equation}
For ease of notation $\ell$ is most often suppressed. Furthermore define the \emph{undershoot}
\begin{equation}\label{underskud}
Z= \ell-X_{\tau}\,,
\end{equation}
which is well--defined on the set $\{\tau <\infty\}$. Note that the level $\ell$ can by crossed through continuity as well as
a result of a downward jump. Of interest is a joint expression about $(\tau<\infty)$ and the distribution of $Z$. This is expressed through the expressions
\begin{equation}\label{laplace}
\EEx[e^{-\zeta Z};A_j]\quad \textrm{and}\quad
\EEx[A_c]\,,
\end{equation}
where  $A_j$ and $A_c$ is a partition of the set $\{\tau<\infty\}$ into
the jump case $A_j=\{\tau<\infty,X_{\tau}<\ell\}$ and the continuity case
$A_c=\{\tau<\infty,X_{\tau}=\ell\}$. The expressions in (\ref{laplace}) can be found from solving two equations
\begin{equation}\label{equations}
\EEx[\ee^{-\zeta Z};A_j]+f_i(\ell)\EEx[A_c]=f_i(x)\,,\quad i=1,2\,,
\end{equation}
where $f_1$ and $f_2$ are partial eigenfunctions for the
infinitesimal generator $\AAA$ for the process: $f_i:\RR\to\CC$ are bounded and differentiable on $[\ell;\infty)$ and satisfy
the condition that
$$
\AAA f_i(x)=0 \quad\textrm{for all }x\in [\ell;\infty)\,,
$$
where $\AAA$ is defined by
\begin{equation}
\AAA f(x)=\kappa x f'(x)+\lambda\int_{\RR}\Big(f(x+y)-f(x)\Big)G(\dd y)\,,
\end{equation}
for details, see \cite{art3}. In addition, each $f_i$ has the following exponential form
on the interval $(-\infty;\ell)$
$$
f_i(x)=\ee^{-\zeta (\ell-x)}\quad \textrm{for }x<\ell\,.
$$
It is important to notice that there exists some situations where only
one partial eigenfunction is needed: If $\ell\kappa>0$ the probability
$\PPx(A_c)$ of crossing $\ell$ through continuity is 0 
(recall that the process is deterministic and monotone between jumps). In this case finding $\EEx[\ee^{-\zeta Z};A_j]$ is even simpler (from (\ref{equations}) with the
$A_c$ part equal 0):
\begin{equation}\label{equation}
\EEx[\ee^{-\zeta Z};A_j]=f(x)\,,
\end{equation}
where $f$ is the single partial eigenfunction.\\[2mm]
In the negative drift case ($\kappa<0$) the recurrence of $X$ gives that
$\PPx(A_j)+\PPx(A_c)=\PPx(\tau<\infty)=1$. If furthermore $\zeta=0$ the desired expressions in (\ref{laplace}) reduce to the probabilities $\PPx(A_j)$ and
$\PPx(A_c)$. Hence, only one partial eigenfunction is needed in order
to solve the equation.

In \cite[Theorem 4]{art3} a result is given that sketches how to
construct such partial eigenfunctions. In the following this theorem
is reformulated in order to fit the further calculations. Define
\begin{equation}\label{fnul}
f_0(y)=\left\{\begin{array}{ll}
0 & y\geq \ell\\
\ee^{-\zeta(\ell-y)} & y<\ell
\end{array}\right.\,,
\end{equation}
and
\begin{equation}\label{fgamma}
f_{\Gamma}(y)=\left\{\begin{array}{ll}
\int_{\Gamma}\psi(z)\ee^{-yz}\,\dd z & y\geq \ell\\
0 & y<\ell
\end{array}\right.\;,
\end{equation}
where $\psi$ is the complex valued kernel defined by
\begin{equation}\label{psinul}
\psi(z)=z^{-1}\Bigg(\prod_{k=1}^r(z-\mu_k)^{-\tfrac{p\lambda\alpha_k}{\kappa}}\Bigg)\Bigg(\prod_{d=1}^s(z+\nu_d)^{-\tfrac{q\lambda\beta_d}{\kappa}}\Bigg)\,,
\end{equation}
and $\Gamma$ is some suitable curve in the complex plane of the form
$\Gamma=\{\gamma(t):\delta_1<t<\delta_2\}$ for $-\infty\leq \delta_1<\delta_2\leq \infty$. The parameters $\alpha_k,\mu_k,\beta_d,\nu_d$ are given in (\ref{fml:gdensity}). Note that
\begin{equation}\label{order}
|\psi(z)|=O\left(|z|^{-1-\lambda/\kappa}\right)\,,
\end{equation}
when $|z|\to\infty$.
\begin{Th}\label{satningfin}
Let $\zeta\geq 0$ be given and let $f_0$ and
$f_{\Gamma_i}$ be defined as in (\ref{fnul}) and (\ref{fgamma}) for
$i=1,\ldots,m$, such that all $\Gamma_i$ are concentrated on the
positive part of the complex plane $\CC_+=\{z\in\CC:\mathrm{Re}z\geq
0\}$. Assume that for each contour $\Gamma_i$ a holomorfic version of $\psi$
exists that contains the contour. Assume furthermore that for $i=1,\ldots,m$ it holds that  
\begin{itemize}
\item[\textrm{(i)}]$\int_{\Gamma_i}|\psi(z)|\ee^{-\ell \mathrm{Re} z}\,\dd z<\infty$
\item[\textrm{(ii)}]$\int_{\Gamma_i}|\psi(z)|\,|z|\ee^{-\ell \mathrm{Re}
    z}\,\dd z<\infty\quad$
\item[\textrm{(iii)}]$\int_{\Gamma_i}|\tfrac{\psi(z)}{z-\mu_k}|\ee^{-\ell
    \mathrm{Re} z}\,\dd z<\infty$
\item[\textrm{(iv)}]$\psi(\gamma_i(\delta_{i1}))\gamma_i(\delta_{i1})\ee^{-y\gamma(\delta_{i1})}=\psi(\gamma_i(\delta_{i2}))\gamma_i(\delta_{i2})\ee^{-y\gamma_i(\delta_{i2})}$\,.
\end{itemize}
Define
\begin{equation}\label{fdef}
f(y)=\sum_{i=1}^mc_if_{\Gamma_i}(y)+f_0(y)\,.
\end{equation}
If the constants $c_1,\ldots,c_m$ are chosen such that
\begin{equation}\label{ligningssystemet}
\sum_{i=1}^mc_iM_{\Gamma_i}^k+\frac{\mu_k}{\mu_k+\zeta}=0
\end{equation}
for $k=1,\ldots,r$ where $M_{ik}$ is given by
$$
M_{\Gamma_i}^k=\mu_k
\int_{\Gamma_i}\frac{\psi(z)}{z-\mu_k}\ee^{-\ell z}\,\dd z
$$
for $i=1,\ldots,m$ and $k=1,\ldots,r$, then $f$ is a partial
eigenfunction for the generator $\AAA$.
\end{Th}
The theorem shows what it takes to construct a partial eigenfunction:
As many $f_{\Gamma_i}$--functions integration contours such that the equation system
(\ref{ligningssystemet}) can be solved. For the construction of one
partial eigenfunction $m=r$ integration contours are needed (note that
the equation system is inhomogeneous and has $m$ unknowns). If an additional
eigenfunction is requested $m=r+1$ different integration contours
should be found. To solve the equation system (\ref{ligningssystemet}) with respect to the unknowns $c_1,\ldots,c_m$
implies that the vectors $M_{\Gamma_i}=(M_{\Gamma_i}^1,\ldots,M_{\Gamma_i}^m)$ for $i=1,\ldots,r$ have to be linearly independent.\\[2mm]
Theorem \ref{satningfin} can be used for all values of $\ell$. However, it
restricts the choice of integration
contours. That makes the following adapted theorem useful. Define two new versions of the
$f_\Gamma$--functions:
\begin{eqnarray}\label{nydefsigma}
f^1_{\Gamma_1}(y)&=&\left\{\begin{array}{ll}
\int_{\Gamma_1}\psi(z)\ee^{-yz}\,\dd z & y> 0\\
0 & y<0
\end{array}\right.\nonumber\\
f^2_{\Gamma_2}(y)&=&\left\{\begin{array}{ll}
\int_{\Gamma_2}\psi(z)\ee^{-yz}\,\dd z & \ell \leq y<0\\
0 & \textrm{otherwise}
\end{array}\right.\;.
\end{eqnarray}
For convenience we shall use the following definitions
\begin{Def}\label{definition}
\begin{eqnarray*}
&\phantom{=}&
M^{1k}_{\Gamma_i}=\int_{\Gamma_{i1}}\frac{\psi(z)}{z-\mu_k}\dd z
\qquad i=1,\ldots,m,\quad k=1,\ldots,r\\
&\phantom{=}&
M^{2d}_{\Gamma_i}=\int_{\Gamma_{i1}}\frac{\psi(z)}{\nu_d+z}\dd z
\qquad i=1,\ldots,m,\quad d=1,\ldots,s\\
&\phantom{=}&
N^{1k}_{\Gamma_j}=\int_{\Gamma_{j2}}\frac{\psi(z)}{\mu_k-z}\,\dd z
\qquad j=1,\ldots,n,\quad k=1,\ldots,r\\
&\phantom{=}&N^{2d}_{\Gamma_j}=\int_{\Gamma_{j2}}\frac{\psi(z)}{\nu_d+z}\dd
z
\qquad j=1,\ldots,n,\quad d=1,\ldots,s\\
&\phantom{=}&N^{3k}_{\Gamma_j}=\int_{\Gamma_{j2}}\frac{\psi(z)}{z-\mu_k}\ee^{-\ell z}\dd
z
\qquad j=1,\ldots,n,\quad k=1,\ldots,r\,.
\end{eqnarray*}
\end{Def}
We will need
\begin{Con}\label{conditions}
Let $\zeta\geq 0$ be given and let $f_0$,
$f^1_{\Gamma_{i1}}$ and $f^2_{\Gamma_{j2}}$ be defined as in (\ref{nydefsigma}) for
$i=1,\ldots,m$ and $j=1,\ldots,n$ such that all $\Gamma_{i1}\subset
\CC_+$ and $\Gamma_{j2}\subset\CC$ are suitable complex curves
($\psi$ should have holomorfic versions containing these
curves). Assume for $\psi$ and $\Gamma_{i1}$, $i=1,\ldots,m$, that
\begin{itemize}
\item[\textrm{(i)}]$ \int_{\Gamma_{i1}}|\psi(z)|\,\dd z<\infty$
\item[\textrm{(ii)}]$\int_{\Gamma_{i1}}|\psi(z)|\,|z|\ee^{-y \mathrm{Re}
    z}\,\dd z<\infty\quad$ for all $y>0$ 
\item[\textrm{(iii)}]$\int_{\Gamma_{i1}}|\tfrac{\psi(z)}{z-\mu_k}|\,\dd z<\infty$
  for $k=1,\ldots,r$
\item[\textrm{(iv)}]$\int_{\Gamma_{i1}}|\tfrac{\psi(z)}{z+\nu_d}|\,\dd z<\infty$
  for $d=1,\ldots,s$
\item[\textrm{(v)}]$
\psi(\gamma_{i1}(\delta_{i1}))\gamma_{i1}(\delta_{i1}^1)\ee^{-y\gamma_{i1}(\delta_{i1}^1)}=\psi(\gamma_{i1}(\delta_{i2}^1))\gamma_{i1}(\delta_{i2}^1)\ee^{-y\gamma_{i1}(\delta_{i2}^1)}$
for all $y>0$\,,
\end{itemize}
and similarly for $\psi$ and $\Gamma_{j2}$ that
\begin{itemize}
\item[\textrm{(i')}]$ \int_{\Gamma_{j2}}|\psi(z)|\,\dd z<\infty$
\item[\textrm{(ii')}]$ \int_{\Gamma_{j2}}|\psi(z)|\ee^{-\ell\textrm{Re}z}\,\dd
  z<\infty$
\item[\textrm{(iii')}]$\int_{\Gamma_{j2}}|\psi(z)|\,|z|\ee^{-y \mathrm{Re}
    z}\,\dd z<\infty\quad$ for all $y\in [\ell;0[$
\item[\textrm{(iv')}]$\int_{\Gamma_{j2}}|\tfrac{\psi(z)}{z-\mu_k}|\,\dd z<\infty$
  for $k=1,\ldots,r$
\item[\textrm{(v')}]$\int_{\Gamma_{j2}}|\tfrac{\psi(z)}{z-\mu_k}|\ee^{-\ell z}\,\dd
  z<\infty$ for $k=1,\ldots,r$
\item[\textrm{(vi')}]$\int_{\Gamma_{j2}}|\tfrac{\psi(z)}{z+\nu_d}|\,\dd z<\infty$
  for $d=1,\ldots,s$
\item[\textrm{(vii')}]$\psi(\gamma_{j2}(\delta_{j1}^2))\gamma_{j2}(\delta_{j1}^2)\ee^{-y\gamma_2(\delta_{j1}^2)}=\psi(\gamma_{j2}(\delta_{j2}^2))\gamma_{j2}(\delta_{j2}^2)\ee^{-y\gamma_{j2}(\delta_{j2}^2)}$\\
  for all $\ell\leq y<0$.
 \end{itemize}
for $j=1,\ldots,n$.
\end{Con}
With these definitions we can state
\begin{Th}\label{storogfin}
Assume that the integration contours $\Gamma_{i1}$, $i=1,\ldots,m$ and
$\Gamma_{j2}$, $j=1,\ldots,n$ satisfy the conditions in Condition
\ref{conditions}. Define $f:\RR\to\CC$ by
\begin{equation}\label{flangsigma}
f(y)=\sum_{i=1}^mc_if^1_{\Gamma_{i1}}(y)+\sum_{j=1}^nb_jf^2_{\Gamma_{j2}}(y)+f_0(y)\,.
\end{equation}
Then $f$ is bounded and differentiable on $\ell\to\infty$. If the
constants $c_1,\ldots,c_m$ and $b_1,\ldots,b_n$ fulfil the equations
\begin{equation}\label{linetsigma}
\sum_{j=1}^nb_j N^{3k}_{\Gamma_j}+\frac{1}{\mu_k+\zeta}=0
\end{equation}
and
\begin{equation}\label{lintosigma}
\Bigg(\sum_{i=1}^mc_iM^{1k}_{\Gamma_i}\Bigg)+\Bigg(\sum_{j=1}^nb_jN^{1k}_{\Gamma_j}\Bigg)=0
\end{equation}
for $k=1,\ldots,r$  together with
\begin{equation}\label{lintresigma}
\Bigg(\sum_{j=1}^nb_jN^{2d}_{\Gamma_j}\Bigg)-\Bigg(\sum_{i=1}^mc_iM^{2d}_{\Gamma_j}\Bigg)=0
\end{equation}
for $d=1,\ldots,s$, then $f$ is a partial eigenfunction for $\AAA$.
\end{Th}
\begin{proof}
As in the proof of \cite[Theorem~4]{art3} it is seen that for $y\geq 0$
\[
\AAA f^1_{\Gamma_{i1}}=p\lambda \sum_{k=1}^r\alpha_k\mu_k\, M^{1k}_{\Gamma_i}\,\ee^{-\mu_k y}
\]
and for $\ell\leq y<0$
\[
\AAA f^1_{\Gamma_{i1}}=-q\lambda \sum_{d=1}^s\beta_d\nu_d\, M^{2d}_{\Gamma_i}\,\ee^{\nu_d y}
\]
Furthermore, we find for $y\geq 0$ that
\[
\AAA f^2_{\Gamma_{j2}}=p\lambda \sum_{k=1}^r\alpha_k\mu_k\, N^{1k}_{\Gamma_j}\,\ee^{-\mu_k y}+p\lambda \sum_{k=1}^r\alpha_k\mu_k \,N^{2k}_{\Gamma_j}\,\ee^{\mu_k \ell}\,\ee^{-\mu_k y}
\]
and finally, for $\ell\leq y<0$
\[
\AAA f^2_{\Gamma_{j2}}=q\lambda \sum_{d=1}^s\beta_d\nu_d\, N^{2d}_{\Gamma_j}\,\ee^{\nu_d y}+p\lambda \sum_{k=1}^r\alpha_k\mu_k \,N^{3k}_{\Gamma_j}\,\ee^{\mu_k\ell}\,\ee^{-\mu_k y}
\]
Since for all $y\geq\ell$
\[
\AAA f_0(y)=\lambda \sum_{k=1}^r\alpha_k\mu_k\frac{1}{\mu_k+\zeta}\,\ee^{\mu_k\ell}\,\ee^{-\mu_k y}\,,
\]
it follows that $\AAA f(y)=0$ for all $y\geq\ell$, if the equations (\ref{linetsigma})--(\ref{lintresigma}) are satisfied.
\end{proof}

\subsection{The choice of integration contours}\label{contours}
There are several possible choices for the integration contours, see  \cite{art3}. The choice described in the following applies to cases with
positive drift $\kappa$ and will differ from the ones
defined in \cite{art3}. The situation $\kappa<0$ is studied in Section \ref{undershoot}.

First assume that $\ell>0$. Then only one partial
eigenfunction is needed and we shall use Theorem
\ref{satningfin}. The definition of the $m=r$ contours has its starting point in the zeros and
singularities of the kernel $\psi$. The real--valued points $-\nu_s,\ldots,-\nu_1,0,\mu_1,\ldots,\mu_r$ from (\ref{fml:gdensity})
are all such zeros or singularities. The
contours $\Gamma_1,\ldots,\Gamma_{r}$ are chosen as follows 
\begin{itemize}
\item If $\mu_i$ is a zero for $\psi$ define
$$
\Gamma_i=\{\mu_i+(1+i)t:0\leq t<\infty\}\,. 
$$
\item If $\mu_i$ is a singularity for $\psi$ define
$$
\Gamma_i=\{\mu+(-1+i)t:-\infty<t\leq
0\}\cup\{\mu+(1+i)t:0\leq t<\infty\}
$$
for a $\mu\in(\mu_{i-1},\mu_i)$ (with the convention $\mu_0=0$).
\end{itemize}
A sketch of the chosen contours can be seen in Figure \ref{Figur1}.
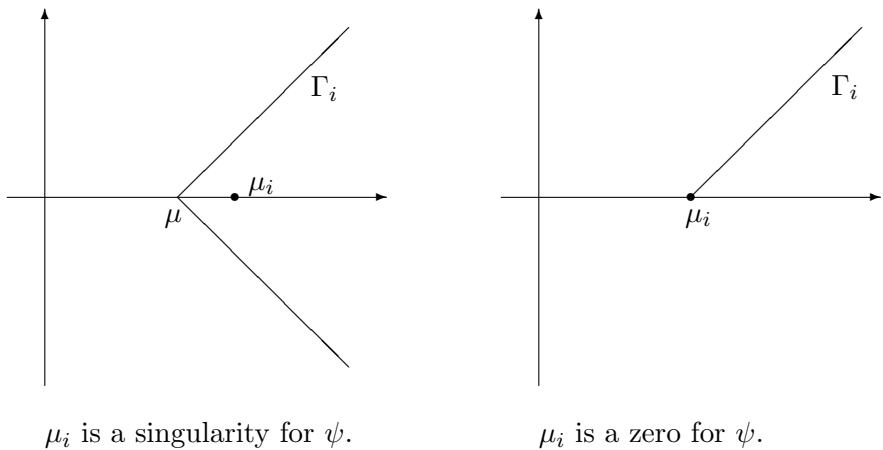
\begin{figure}
\setlength{\unitlength}{5mm}
\begin{center}
\begin{picture}(24,11)(-1,-6.5)
\put(-1,0){\vector(1,0){10}}
\put(0,-5){\vector(0,1){10}}
\put(5,0){\circle*{0.2}}
\put(3.5,0){\line(1,1){4.5}}
\put(3.5,0){\line(1,-1){4.5}}
\put(0,-6.5){$\mu_i$ is a singularity for $\psi$.}
\put(5.35,0.2){$\mu_i$}
\put(3.15,-0.65){$\mu$}
\put(7,2.68){$\Gamma_{i}$}

\put(12,0){\vector(1,0){10}}
\put(13,-5){\vector(0,1){10}}
\put(17,0){\circle*{0.2}}
\put(17,0){\line(1,1){4.5}}
\put(13,-6.5){$\mu_i$ is a zero for $\psi$.}
\put(16.85,-0.65){$\mu_i$}
\put(20.7,2.72){$\Gamma_{i}$}
\end{picture}
\caption{The contour $\Gamma_i$ in the two cases: $\mu_i$ is a
  singularity (left) for $\psi$ and $\mu_i$ is a zero (right)}
\label{Figur1}
\end{center} 
\end{figure}
Next assume that $\ell<0$. Then Theorem~\ref{storogfin} is
used. For the contours $\Gamma_{11},\ldots,\Gamma_{r1}$ one can use
$\Gamma_1,\ldots,\Gamma_r$ from above. It remains to find
$n=r+s+1$ contours $\Gamma_{12},\ldots,\Gamma_{r+s+1,2}$ in order to
construct two eigenfunctions. For convenience let $p_1,\ldots,p_{r+s+1}$ denote the points
$-\nu_s,\ldots,-\nu_1,0,\mu_1,\ldots,\mu_r$ and use the following recipe:
\begin{itemize}
\item If $p_i$ is a zero for $\psi$ define
$$
\Gamma_{i2}=\{p_i+(-1+i)t:0\leq t<\infty\} 
$$
\item If $p_i$ is a singularity for $\psi$ define
$$
\Gamma_{i2}=\{p+(1+i)t:-\infty<t\leq
0\}\cup\{p+(-1+i)t:0\leq t<\infty\} 
$$
for a $p\in(p_{i};p_{i+1})$ (with the convention $p_{r+s+2}=\infty$).
\end{itemize}
\begin{Rem}\label{kompleksversion}
For the contours $\Gamma_{i}$ corresponding to a singularity the specific choice of $\mu$ in $(\mu_{i-1},\mu_i)$ is
without influence as a result of Cauchy's Theorem. In fact, $\mu$ can
be chosen freely in $(\mu_l,\mu_i)$ where $\mu_l$ is the largest
singularity for $\psi$ less than $\mu_i$ (remember that 0 is a
singularity so that $\mu_l\geq 0$). Moreover, it can never happen that
$f_{\Gamma_i}=f_{\Gamma_{i+1}}$ in the case where both $\mu_i$ and
  $\mu_{i+1}$ are singularities. If $\mu_i$, the singularity that
  separates the two contours, is of order $\rho<0$ with $\rho\notin
  \ZZ$ this is secured from the use of different versions of the
  complex logarithm in the respective domains of the contours. If the singularity $\mu_i$ is an
  integer the argument that $f_{\Gamma_i}\neq f_{\Gamma_{i+1}}$ is
  based on Cauchy's Theorem.
\end{Rem} 

\section{Asymptotics of the ruin probability as $x\to\infty$}\label{xasympt}
When the drift $\kappa>0$ then $\PPx(\tau<\infty)<1$. Furthermore, the probability decreases when the initial value $x$
increases. Solving the equation system (\ref{equations})
w.r.t. $\PPx(\tau<\infty)=\PPx(A_c)+\PPx(A_j)$ we have for $\ell<0$
\begin{equation}\label{ruinssh}
\PPx(\tau<\infty)=f_1(x)\frac{1-f_2(\ell)}{f_1(\ell)-f_2(\ell)}+f_2(x)\frac{f_1(\ell)-1}{f_1(\ell)-f_2(\ell)}\,,
\end{equation}
where $f_1$ and $f_2$ are the two partial eigenfunctions
constructed in Theorem \ref{storogfin}. When $\ell>0$ we have
$$
\EEx[A_j]=f(x)\,,
$$
where $f$ is the single eigenfunction constructed in Theorem~\ref{satningfin}.
It is essential that the construction of the partial eigenfunctions
$f_1$ and $f_2$ (or $f$ in the $\ell>0$ case) does not depend on $x$. The
behaviour of the probability $\PPx(\tau <\infty)$ to be studied is
therefore only determined by the behaviour of the two partial
eigenfunctions $f_1$ and $f_2$ when $x\to\infty$. We have the
following result:
\begin{Th}\label{voksendex}
There exists a constant $K$ such that
$$
\lim_{x\to\infty}\frac{\PPx(\tau<\infty)}{\ee^{-\mu_1
    x}x^{-\frac{p\alpha_1\lambda}{\kappa}-1}}=K\,.
$$
The constant $K$ is expressed explicitly in (\ref{konstanten}) below
when $\ell<0$ and in (\ref{konstanten2}) when $\ell>0$.
\end{Th}
For the later use of the results it is convenient to formulate part of
the proof of Theorem \ref{voksendex} as self--contained lemmas. Furthermore, the definitions $\rho_j=-p\alpha_j\lambda/\kappa$ and
$$
\psi_{\setminus\{\mu_{j}\}}(z)=z^{-1}\left(\prod_{k=1,k\neq j}^r(z-\mu_k)^{-\frac{p\alpha_k\lambda}{\kappa}}\right)\left(\prod_{d=1}^s(z+\nu_d)^{-\frac{q\beta_d\lambda}{\kappa}}\right)\,.
$$
for $j=1,\ldots,r$ will be convenient. Now $f^1_{\Gamma_{j1}}$ can be written as
$$
f^1_{\Gamma_{j1}}(x)=\int_{\Gamma}(z-\mu_{j})^{\rho_j}\psi_{\setminus\{\mu_j\}}(z)\ee^{-xz}\,\dd z\,.
$$ 
The first lemma concerns the case, where $\alpha_j<0$. Here $\mu_j$ is a zero for $\psi$, and $\Gamma_{j1}=\{\mu_j+(1+i)t\,:\,0\leq t<\infty\}$. We find
\begin{Lemma}\label{hjaelpeet} 
Assume $\alpha_j<0$. Then it holds that
\begin{equation}\label{graense1}
\lim_{x\to\infty}\frac{f^1_{\Gamma_{j1}}(x)}{\ee^{-\mu_{j} x}x^{\rho-1}}=\psi_{\setminus\{\mu_j\}}(\mu_{j})\int_{\Gamma_0}z^{\rho_j}\ee^{-z}\,\dd z\,,
\end{equation}
where $\Gamma_0$ is the integration contour
\begin{equation}\label{gammanul}
\Gamma_0= \{(1+i)t:0\leq t<\infty\}\,.
\end{equation}
\end{Lemma}
\begin{proof}
The expression of $f^1_{\Gamma_{j1}}(x)$ can be rewritten in the following
way
\begin{align}\label{integral}
f^1_{\Gamma_{j1}}(x)&=\int_{\Gamma_{j1}}(z-\mu_{j})^{\rho_j}\psi_{\setminus\{\mu_j\}}(z)\ee^{-xz}\,\dd z\nonumber\\
&=\int_0^{\infty}(1+i)\big((1+i)t\big)^{\rho_j}\psi_{\setminus\{\mu_j\}}\big(\mu_{j}+(1+i)t\big)\ee^{-x(\mu_{j}+(1+i)t)}\,\dd
t\nonumber\\
&=x^{-{\rho_j}-1}\ee^{-\mu_{j}
  x}\int_0^{\infty}(1+i)\big((1+i)s\big)^{\rho_j}\psi_{\setminus\{\mu_j\}}\big(\mu_{j}+(1+i)\tfrac{s}{x}\big)\ee^{-s(1+i)}\,\dd s\,,
\end{align}
where the substitution $s=tx$ has been used. Consider the function $t\mapsto |\psi_{\setminus\{\mu_j\}}\big(\mu_{j}+(1+i)t\big)|$, which is continuous and strictly positive. Furthermore it is $O(|\mu_{j}+(1+2i)t|^{-1-\lambda/\kappa-\rho_j})$, when
$t\to\infty$. This gives the existence of a constant $C<\infty$ such
that
$$
|\psi_{\setminus\{\mu_j\}}\big(\mu_{j}+(1+i)t\big)|\leq C\quad \textrm{for all }t\geq 0\,.
$$
In particular, this holds when $t=s/x$ for all $s\geq 0$ and
$x>0$. Thus, the function
$$
s\mapsto C|(1+i)((1+2i)s)^{\rho_j}|\ee^{-s}
$$
is an integrable upper bound for the integrand in the last line
of (\ref{integral}). By dominated convergence we get that
\begin{align*}
&\phantom{=}\lim_{x\to\infty}
\int_0^{\infty}(1+i)\big((1+i)s\big)^{\rho_j}\psi_{\setminus\{\mu_j\}}\big(\mu_{j}+(1+i)\tfrac{s}{x}\big)\ee^{-s(1+i)}\,\dd
s\\
&=\int_0^{\infty}(1+i)\big((1+i)s\big)^{\rho_j}\psi_{\setminus\{\mu_j\}}(\mu_{j})\ee^{-s(1+i))}\,\dd
s\\
&=\psi_{\setminus\{\mu_j\}}(\mu_{j})\int_{\Gamma_0}z^{\rho_j}\ee^{-z}\,\dd z\,.
\end{align*}
Hence the result is shown.
\end{proof}
For the proof of the next lemma we define
$$
\Gamma_{\mu}=\{\mu+(-1+i)t:-\infty<t\leq 0\}\cup\{\mu+(1+i)t:0<t<\infty\}\,,
$$
for $\mu>0$. Note that if $\alpha_j>0$, then $\mu_j$ is a singularity for $\psi$ and $\Gamma_{j1}=\Gamma_\mu$, where $\mu\in(\mu_{j-1},\mu_j)$. We have
\begin{Lemma}\label{hjaelpeto}
Assume that $\alpha_j>0$. Then
\begin{equation}\label{reshjaelpeto}
\lim_{x\to\infty}\frac{f^1_{\Gamma_{j1}}(x)}{x^{\rho_j-1}e^{-\mu_j
    x}}=\psi_{\setminus\{\mu_j\}}(\mu_j)\int_{\Gamma_{-a}}z^{\rho_j}\ee^{-z}\,\dd z\,,
\end{equation}
where
$$
\Gamma_{-a} = \{-a+(-1+i)t:-\infty<t\leq 0\}\cup\{-a+(1+i)t:0<t<\infty\}
$$
and $a>0$ is any positive real number.
\end{Lemma}
\begin{proof}
In Remark \ref{kompleksversion} it was argued that
$$
f^1_{\Gamma_{j1}}(x)=f^1_{\Gamma_{\mu'}}(x)
$$
for all $\mu'\in(\mu_{l},\mu_j)$, where $\mu_l$ is the largest singularity for $\psi$ less than $\mu_j$. We choose $\mu'=\mu_j-\tfrac{a}{x}$ for some suitable
$a>0$. Hence,
\begin{align*}
&f^1_{\Gamma_{j1}}(x)\\
&= f^1_{\Gamma_{\mu_j-a/x}}(x)\nonumber\\
&=\int_0^{\infty}(1+i)\left(-\tfrac{a}{x}+(1+i)t\right)^{\rho}\psi_{\setminus\{\mu_j\}}\big(\mu_j-\tfrac{a}{x}+(1+i)t\big)\ee^{-x\mu_j+a-x(1+i)t}\,\dd
t\\
&+\int_{-\infty}^0(-1+i)\left(-\tfrac{a}{x}+(-1+i)t\right)^{\rho}\psi_{\setminus\{\mu_j\}}\big(\mu_j-\tfrac{a}{x}+(-1+i)t\big)\ee^{-x\mu_j+a-x(-1+i)t}\,\dd
t\,.
\end{align*}
Using the substitution $s=tx$ yields that the first
integral equals
\begin{equation}\label{nytigen}
x^{\rho-1}\ee^{-\mu_j
  x}\int_0^{\infty}(1+2i)\big((1+i)s-a\big)^{\rho_j}\psi_{\setminus\{\mu_j\}}\big(\mu_j-\tfrac{a}{x}+(1+i)\tfrac{s}{x}\big)\ee^{a-(1+i)s}\,\dd s\,.
\end{equation}
From dominated convergence the limit of the
integral in (\ref{nytigen}) as $x\to\infty$ is
$$
\psi_{\setminus\{\mu_j\}}(\mu_j)\int_0^{\infty}\big((1+i)s-a\big)^{\rho_j}\ee^{-(1+i)s}\,\dd s\,.
$$
A similar result holds for the second integral. Hence, it has been
shown that
\begin{align}\label{graense2}
&\lim_{x\to\infty}\frac{f^1_{\Gamma_{j1}}(x)}{x^{\rho_j-1}e^{-\mu_j
    x}}\nonumber\\
&=\psi_{\setminus\{\mu_j\}}(\mu_j)\int_0^{\infty}(1+i)\big(-a+(1+i)s\big)^{\rho_j}\ee^{-(-a+(1+i)s)}\,\dd
s\nonumber\\
&\phantom{=}+\psi_{\setminus\{\mu_j\}}(\mu_j)\int_{-\infty}^{0}(-1+i)\big(-a+(-1+i)s\big)^{\rho_j}\ee^{-(-a+(-1+i)s)}\,\dd
s\nonumber\\
&=\psi_{\setminus\{\mu_j\}}(\mu_j)\int_{\Gamma_{-a}}z^{\rho_j}\ee^{-z}\,\dd
z\,.
\end{align}
\end{proof}
\begin{Rem}
The starting point of the contour, $\mu'$, was set to move right
towards $\mu_j$. Another solution could be letting it move left
towards $\mu_{l}$ (the largest singularity less than $\mu_j$) with the definition $\mu'=\mu_l+\tfrac{a}{x}$. From
redoing all the arguments the following result would be reached:
$$
\lim_{x\to\infty}\frac{f^1_{\Gamma_{j1}}(x)}{x^{\rho_l-1}\ee^{-\mu_{l}
    x}}=\phi(\mu_{l})\pi(\mu_{l})\int_{\Gamma_{a}}z^{\rho_l}\ee^{-z}\,\dd z
$$
what appears to be a slower decrease towards 0. However, note that only one of the integrals is different from 0:
$$
\int_{\Gamma_{a}}z^{-\rho_l}\ee^{-z}\,\dd z=0 \quad \textrm{and}\quad\int_{\Gamma_{-a}}z^{-\rho_j}\ee^{-z}\,\dd z\neq 0\,.
$$
\end{Rem}
\begin{proof}[Proof of Theorem \ref{voksendex}] 
Assume $\ell<0$ (if $\ell>0$ the calculations will be simpler). Both $f_1$ and $f_2$ are linear combinations of the
$f_{\Gamma}$ functions. Since $x$ is assumed to be positive all
$f^2_{\Gamma_{j2}}(x)=0$. Then $f_1(x)$ and $f_2(x)$ are linear combinations of
$$
f^1_{\Gamma_{11}}(x),\ldots, f^1_{\Gamma_{m1}}(x)\,.
$$
So in order to study $\PPx(\tau<\infty)$ it is sufficient to determine the behaviour
of the functions $f^1_{\Gamma_{i1}}(x)$, when $x\to\infty$. For each each $i=1,\ldots,r$ there are two possible situations to
consider: $\alpha_i<0$ or $\alpha_i>0$. It was shown in
Lemma~\ref{hjaelpeet} and Lemma~\ref{hjaelpeto} that either way 
$$
\lim_{x\to\infty}\frac{f^1_{\Gamma_{i1}}(x)}{x^{\rho_i-1}e^{-\mu_i
    x}}=K_i
$$
for some constant $K_i$. Since the ruin probability $\PPx(\tau<\infty)$ can be written as a
linear combination of these functions, the asymptotics are determined by the function with the slowest
decrease. This is $f^1_{\Gamma_{11}}$, and since $\mu_1$ is always a singularity for $\psi$,
the exact asymptotic behaviour of $f^1_{\Gamma_{11}}$ can be found in Lemma
\ref{hjaelpeto}. 

Let the two partial eigenfunctions $f_1$ and $f_2$ be the linear
combinations
\begin{equation}\label{linearkomb}
f_1(x)=\sum_{i=1}^{r}c_i^1f^1_{\Gamma_{1i}}(x) \quad
\textrm{and}\quad f_2(x)=\sum_{i=1}^{r}c_i^2f^1_{\Gamma_{1i}}(x)
\end{equation}   
for $x>0$. Then 
\begin{align*}
&\lim_{x\to\infty}\frac{\PPx(\tau<\infty)}{\ee^{-\mu_1
    x}x^{-\frac{p\alpha_l\lambda}{\kappa}-1}}\\
&=
\lim_{x\to\infty}\frac{f^1_{\Gamma_{\mu_1}}(x)}{\ee^{-p_{\mu_1} x}x^{-\frac{p\alpha_l\lambda}{\kappa}-1}}\left(c_{1}^1\frac{1-f_2(\ell)}{f_1(\ell)-f_2(\ell)}+c_{1}^2\frac{f_1(\ell)-1}{f_1(\ell)-f_2(\ell)}\right)=K\,,
\end{align*}
where $K$ is given by
\begin{align}\label{konstanten}
K=&\left(\psi_{\setminus\{\mu_{1}\}}(\mu_1)\int_{\Gamma_{-a}}z^{\frac{p\alpha_l\lambda}{\kappa}}\ee^{-z}\,\dd
  z\right)\times\nonumber\\
&\left(c_{1}^1\frac{1-f_2(\ell)}{f_1(\ell)-f_2(\ell)}+c_{1}^2\frac{f_1(\ell)-1}{f_1(\ell)-f_2(\ell)}\right)\,.
\end{align}
Hence, the theorem is proved for $\ell<0$. With the same arguments for $\ell>0$ we derive
\begin{equation}\label{konstanten2}
K=c_1\left(\psi_{\setminus\{\mu_{1}\}}(\mu_1)\int_{\Gamma_{-a}}z^{\frac{p\alpha_l\lambda}{\kappa}}\ee^{-z}\,\dd
  z\right)\,.
\end{equation}
\end{proof}

\section{Asymptotics as $\ell\to-\infty$}\label{lasympt}
The setup for $\ell\to-\infty$
becomes more complicated, since the constants $c_1,\ldots,c_m$ and
$b_1,\ldots,b_n$ in the construction of the partial eigenfunctions
depend on $\ell$.

\subsection{Asymptotics of the ruin probability, positive drift}

To study $\PPx(\tau(\ell)<\infty)$ given by
(\ref{ruinssh}) both $f_i(x)$ and $f_i(\ell)$, $i=1,2$, are needed. For
$x>0$, $\ell<0$ and $i=1$ the expressions are
\begin{eqnarray*}
f_1(\ell)&=&\sum_{j=-s}^{r-1}b_j(\ell)f_{\Gamma_{j,2}}^2(\ell)\\
f_1(x)&=&\sum_{i=1}^rc_i(\ell)f^1_{\Gamma_{i,1}}(x)\,.
\end{eqnarray*}
This definition excludes the last of the integration contours
$\Gamma_{-s,2},\ldots,\Gamma_{r,2}$. Similarly, $f_2(\ell)$ and
$f_2(x)$ are defined by
\begin{eqnarray*}
f_2(\ell)&=&\sum_{j=-s+1}^{r}\tilde{b}_j(\ell)f_{\Gamma_{j,2}}^2(\ell)\\
f_2(x)&=&\sum_{i=1}^r\tilde{c}_i(\ell)f^1_{\Gamma_{i,1}}(x)\,,
\end{eqnarray*}
excluding the first of the contours $\Gamma_{-s,2},\ldots,\Gamma_{r,2}$. The constants $c_1(\ell),\ldots,c_{r}(\ell)$ and $b_{-s}(\ell),\ldots,b_{r-1}(\ell)$ are found as the solution to a
linear equation:
\begin{equation}\label{matrixligning}
\left[\begin{array}{cccccc}
0 & \ldots & 0 & N^{31}_{\Gamma_{-s}}(\ell) & \ldots & N^{31}_{\Gamma_{r-1}}(\ell)\\
\vdots  & \ddots & \vdots & \vdots & \ddots &\vdots \\
0 & \ldots & 0 & N^{3r}_{\Gamma_{-s}}(\ell) & \ldots &
N^{3r}_{\Gamma_{r-1}}(\ell)\\
M^{11}_{\Gamma_1} & \ldots & M^{11}_{\Gamma_r} & N^{11}_{\Gamma_{-s}} & \ldots &
N^{11}_{\Gamma_{r-1}}\\
\vdots  & \ddots & \vdots & \vdots & \ddots &\vdots \\
M^{1r}_{\Gamma_1} & \ldots & M^{1r}_{\Gamma_r} & N^{1r}_{\Gamma_{-s}} & \ldots &
N^{1r}_{\Gamma_{r-1}}\\
-M^{21}_{\Gamma_1} & \ldots & -M^{21}_{\Gamma_r} & N^{21}_{\Gamma_{-s}} & \ldots &
N^{21}_{\Gamma_{r-1}}\\
\vdots  & \ddots & \vdots & \vdots & \ddots &\vdots \\
-M^{2s}_{\Gamma_1} & \ldots & -M^{2s}_{\Gamma_r} & N^{2s}_{\Gamma_{-s}} & \ldots &
N^{2s}_{\Gamma_{r-1}}
\end{array}\right]\left[\begin{array}{c}
c_1(\ell)\\
\vdots\\
c_r(\ell)\\
b_{-s}(\ell)\\
\vdots\\
b_{r-1}(\ell)
\end{array}\right]=\left[\begin{array}{c}
\frac{1}{\mu_1}\\
\vdots\\
\frac{1}{\mu_r}\\
0\\
\vdots\\
0\end{array}\right] \,,
\end{equation}
where we denote the first matrix by $A(\ell)$. The limit of $\PPx(\tau(\ell)<\infty)$ when $\ell \to-\infty$ can then be derived.
\begin{Th}\label{thm:limitresult}
The limits $c_i=\lim_{\ell\to-\infty}c_i(\ell)$ are well defined and
non--zero for $i=1,\ldots,r$, and
\begin{equation}\label{fml:limitresult}
\lim_{\ell \to-\infty}\PPx(\tau(\ell)<\infty)=-\sum_{i=1}^rc_if^1_{\Gamma_{i,1}}(x)\,.
\end{equation}
The $c_i$ constants are found in the Corollary \ref{korollaret} below.
\end{Th}
\begin{Ex}
\begin{figure}[htbp]
\centering
\includegraphics[width=0.8\textwidth]{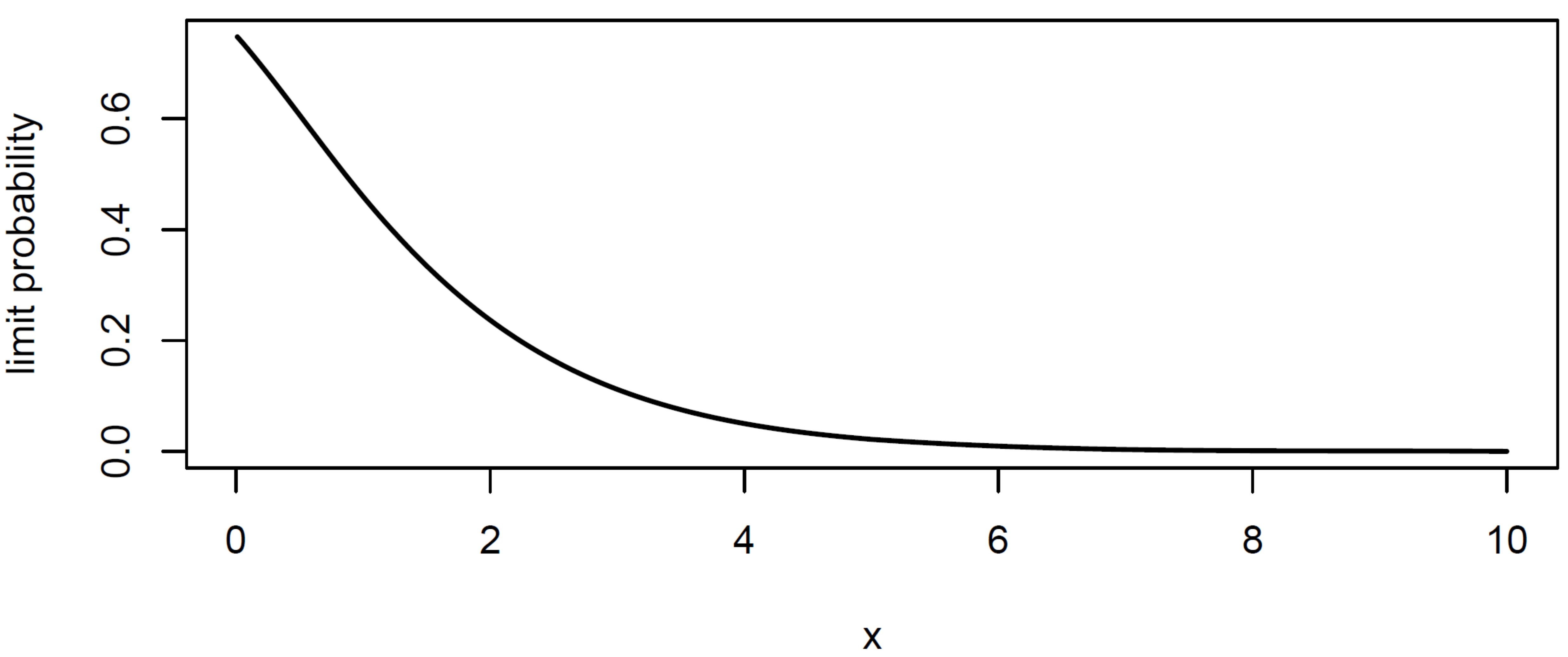}
\caption{Shows $\lim_{\ell \to-\infty}\PPx(\tau(\ell)<\infty)$ as a function of $x$.}
\label{fig2}
\end{figure}
Assume that $r=s=1$, $\kappa=1$, $p=2/3$, $q=1/3$ and $\mu=\nu=1$. Then the limit in (\ref{fml:limitresult}) is a decreasing function of $x$ as illustrated in Figure~\ref{fig2}
\end{Ex}
\begin{proof}[Proof of Theorem~\ref{thm:limitresult}]
Notation: In the proof we will write $f(\ell)=O(g(\ell))$ if there exists a constant $C$ such that $f(\ell)\sim Cg(\ell)$.

In the matrix $A(\ell)$ only
$N^{3k}_{\Gamma_{j}}(\ell)$ (for $k=1,\ldots,r$ and $j=-s,\ldots,r-1$) depends on
$\ell$. Exploring this dependence by applying the same technique as in the
$x\to\infty$ case yields for $k=1,\ldots,r$ and $i=-s,\ldots,-1$ that
\begin{eqnarray}\label{formelet}
\lim_{\ell\to-\infty}\frac{N^{3k}_{\Gamma_i}(\ell)}{\ee^{\ell\nu_{-i}}(-\ell)^{\frac{q\beta_{-i}\lambda}{\kappa}-1}}&=&\lim_{\ell\to-\infty}\frac{1}{\ee^{\ell\nu_{-i}}(-\ell)^{\frac{q\beta_{-i}\lambda}{\kappa}-1}}\int_{\Gamma_{i,2}}\frac{\psi(z)}{z-\mu_k}\ee^{-\ell z}\;\dd
z\nonumber\\
&=& \frac{\psi_{\setminus
    \{-\nu_{-i}\}}(-\nu_{-i})}{-\nu_{-i}-\mu_k}\int_{\tilde{\Gamma}}z^{-\frac{q\beta_{-i}\lambda}{\kappa}}\ee^z\;\dd z
\end{eqnarray}
if $-\nu_{-i}$ is a zero for $\psi$. Here
$$
\tilde{\Gamma}=\{(-1+i)t:0\leq t<\infty\}
$$
and
$$
\psi_{\setminus\{-\nu_{-i}\}}=z^{-1}\left(\prod_{k=1}^r(z-\mu_k)^{-\frac{p\alpha_k\lambda}{\kappa}}\right)\left(\prod_{d=1,d\neq
    i}^s(z+\nu_d)^{-\frac{q\beta_d\lambda}{\kappa}}\right)\,.
$$
If $-\nu_{-i}$ is a singularity the result is 
\begin{equation}\label{formelto}
\lim_{\ell\to-\infty}\frac{N^{3k}_{\Gamma_i}(\ell)}{\ee^{\ell\nu_{-i}}(-\ell)^{\frac{q\beta_{-i}\lambda}{\kappa}-1}}= \frac{\psi_{\setminus
    \{-\nu_{-i}\}}(-\nu_{-i})}{-\nu_{-i}-\mu_k}\int_{\tilde{\Gamma}_a}z^{-\frac{q\beta_{-i}\lambda}{\kappa}}\ee^z\;\dd z\,,
\end{equation}
where 
$$
\tilde{\Gamma}_a=\{a+(1+i)t:-\infty< t\leq 0\}+\{a+(-1+i)t:0\leq t< \infty\}\,.
$$
for any $a>0$. Furthermore 
\begin{equation}\label{formeltre}
\lim_{\ell\to-\infty}N_{\Gamma_0}^{3k}(\ell)=\frac{\psi_{\setminus\{0\}}(0)}{-\mu_k}\int_{\tilde{\Gamma}_a}z^{-1}\ee^z\,\dd
z\,.
\end{equation}
Finally, the constants related to $\mu_1,\ldots,\mu_{r}$ satisfy
the following if $\mu_i$ is a zero
\begin{eqnarray}
\lim_{\ell\to-\infty}\frac{N_{\Gamma_i}^{3i}(\ell)}{\ee^{-\ell\mu_i}(-\ell)^{-\frac{p\alpha_i\lambda}{\kappa}}}&=&
\psi_{\setminus\{\mu_i\}}(\mu_i)\int_{\tilde{\Gamma}}z^{-\frac{p\alpha_i\lambda}{\kappa}-1}\ee^z\;\dd
z\label{formelfire}\\
\lim_{\ell\to-\infty}\frac{N_{\Gamma_i}^{3k}(\ell)}{\ee^{-\ell\mu_i}(-\ell)^{-\frac{p\alpha_i\lambda}{\kappa}-1}}&=&
\frac{\psi_{\setminus\{\mu_i\}}(\mu_i)}{\mu_i-\mu_k}\int_{\tilde{\Gamma}}z^{-\frac{p\alpha_i\lambda}{\kappa}}\ee^z\;\dd
z
\quad \textrm{ if }k\neq i\label{formelfem}
\end{eqnarray}
and if it is a singularity
\begin{eqnarray}
\lim_{\ell\to-\infty}\frac{N_{\Gamma_i}^{3i}(\ell)}{\ee^{-\ell\mu_i}(-\ell)^{-\frac{p\alpha_i\lambda}{\kappa}}}&=&
\psi_{\setminus\{\mu_i\}}(\mu_i)\int_{\tilde{\Gamma}_a}z^{-\frac{p\alpha_i\lambda}{\kappa}-1}\ee^z\;\dd
z\label{formelseks}\\
\lim_{\ell\to-\infty}\frac{N_{\Gamma_i}^{3k}(\ell)}{\ee^{-\ell\mu_i}(-\ell)^{-\frac{p\alpha_i\lambda}{\kappa}-1}}&=&
\frac{\psi_{\setminus\{\mu_i\}}(\mu_i)}{\mu_i-\mu_k}\int_{\tilde{\Gamma}_a}z^{-\frac{p\alpha_i\lambda}{\kappa}}\ee^z\;\dd
z
\quad \textrm{ if }k\neq i\,.\label{formelsyv}
\end{eqnarray}
When calculating the determinant of $A(\ell)$ it is crucial that
$N^{3k}_{\Gamma_i}(\ell)$ has the largest rate of growth when $k=i$. Furthermore, if $\mu_i$  is a singularity of an order in
$(0,1)$ and $k\neq i$ then the limit integral for $N^{3k}_{\Gamma_i}(\ell)$
is zero while the integral in the limit of $N^{3i}_{\Gamma_i}(\ell)$ is not. Define the matrices
$$
M=\left[\begin{array}{cccccc}
M^{11}_{\Gamma_1} & \ldots & M^{11}_{\Gamma_r} & N^{11}_{\Gamma_{-s}} & \ldots &
N^{11}_{\Gamma_{-1}}\\
\vdots  & \ddots & \vdots & \vdots & \ddots &\vdots \\
M^{1r}_{\Gamma_1} & \ldots & M^{1r}_{\Gamma_r} & N^{1r}_{\Gamma_{-s}} & \ldots &
N^{1r}_{\Gamma_{-1}}\\
-M^{21}_{\Gamma_1} & \ldots & -M^{21}_{\Gamma_r} & N^{21}_{\Gamma_{-s}} & \ldots &
N^{21}_{\Gamma_{-1}}\\
\vdots  & \ddots & \vdots & \vdots & \ddots &\vdots \\
-M^{2s}_{\Gamma_1} & \ldots & -M^{2s}_{\Gamma_r} & N^{2s}_{\Gamma_{-s}} & \ldots &
N^{2s}_{\Gamma_{-1}}
\end{array}\right]
$$
and 
$$
N(\ell)=\left[\begin{array}{ccc}
N_{\Gamma_0}^{31}(\ell)& \hdots & N_{\Gamma_{r-1}}^{31}(\ell)\\
\vdots & \ddots & \vdots\\
N_{\Gamma_0}^{3r}(\ell)& \hdots & N_{\Gamma_{r-1}}^{3r}(\ell)\\
\end{array}\right]\;. 
$$
The formulas (\ref{formelet}) -- (\ref{formelsyv}) yield that $\det(A(\ell))\sim\big(\det(N(\ell))(-1)^{r+s+1}\det(M)\big)$
and by using that $N^{3i}_{\Gamma_i}(\ell)$ has the most rapid growth compared
  to $N^{3k}_{\Gamma_i}(\ell)$ when $k\neq i$, it is seen that
$$
\det(N(\ell))\sim\left(N_{\Gamma_0}^{3r}(\ell)\prod_{i=1}^{r-1}N_{\Gamma_i}^{3i}(\ell)\right)
$$
which implies that
$$
\det(N(\ell))=O\left(\ee^{\ell\sum_{j=1}^{r-1}\mu_j}(-\ell)^{\sum_{j=1}^{r-1}\frac{p\alpha_j\lambda}{\kappa}}\right)\,.
$$
Cramer's Rule provides the constants $c_1(\ell),\ldots,c_r(\ell)$ and $b_{-s}(\ell),\ldots,b_{r-1}(\ell)$ in the
equation system (\ref{matrixligning}):
\[
c_1(\ell)=\frac{\det(A_1(\ell))}{\det(A(\ell))}\,,
\]
where
\[
A_1(\ell)=\left[\begin{array}{ccccccc}
\frac{1}{\mu_1}& 0 & \ldots & 0 & N^{31}_{\Gamma_{-s}}(\ell) & \ldots & N^{31}_{\Gamma_{r-1}}(\ell)\\
\vdots & \vdots  & \ddots & \vdots & \vdots & \ddots &\vdots \\
\frac{1}{\mu_r}& 0 & \ldots & 0 & N^{3r}_{\Gamma_{-s}}(\ell) & \ldots &
N^{3r}_{\Gamma_{r-1}}(\ell)\\
0 & M^{11}_{\Gamma_2} & \ldots & M^{11}_{\Gamma_r} & N^{11}_{\Gamma_{-s}} & \ldots &
N^{11}_{\Gamma_{r-1}}\\
\vdots & \vdots  & \ddots & \vdots & \vdots & \ddots &\vdots \\
0& M^{1r}_{\Gamma_2} & \ldots & M^{1r}_{\Gamma_r} & N^{1r}_{\Gamma_{-s}} & \ldots &
N^{1r}_{\Gamma_{r-1}}\\
0& -M^{21}_{\Gamma_2} & \ldots & -M^{21}_{\Gamma_r} & N^{21}_{\Gamma_{-s}} & \ldots &
N^{21}_{\Gamma_{r-1}}\\
\vdots & \vdots & \ddots & \vdots & \vdots & \ddots &\vdots \\
0& -M^{2s}_{\Gamma_2} & \ldots & -M^{2s}_{\Gamma_r} & N^{2s}_{\Gamma_{-s}} & \ldots &
N^{2s}_{\Gamma_{r-1}}
\end{array}\right]\,,
\]
and similarly for the remaining constants. It is seen that 
$$
\det(A_i(\ell))=O\left(\ee^{\ell\sum_{j=1}^{r-1}\mu_j}(-\ell)^{\sum_{j=1}^{r-1}\frac{p\alpha_j\lambda}{\kappa}}\right)
$$ 
for $i=1,\ldots,r+s$ and therefore
\begin{eqnarray*}
c_i(\ell)&=& \frac{\det(A_i(\ell))}{\det(A(\ell))}=O(1)\qquad i=1,\ldots,r\\
b_j(\ell)&=& \frac{\det(A_{r+s+1+j}(\ell))}{\det(A(\ell))}=O(1) \qquad j=-s,\ldots,-1\,.
\end{eqnarray*}
Furthermore, 
\begin{eqnarray*}
\det(A_{r+s+1}(\ell))&\sim&\left(\det(M)\times \frac{1}{\mu_r}\prod_{i=1}^{r-1}N_{\Gamma_i}^{3i}(\ell)\right)\\
\det(A_{r+s+1+j}(\ell))&\sim&\left(\det(M)\times
  \frac{1}{\mu_j}N_{\Gamma_0}^{3r}(\ell)\prod_{i=1,i\neq
    j}^{r-1}N_{\Gamma_i}^{3i}(\ell)\right)\qquad j=1,\ldots,r-1
\end{eqnarray*}
such that
\begin{eqnarray*}
b_0(\ell)&=&\frac{\det(A_{r+s+1}(\ell))}{\det(A(\ell))}\sim\left(\frac{1}{\mu_r}\frac{1}{N^{3r}_{\Gamma_0}(\ell)}\right)\\
b_j(\ell)&=&\frac{\det(A_{r+s+1+j}(\ell))}{\det(A(\ell))}\sim\left(\frac{1}{\mu_j}\frac{1}{N^{3j}_{\Gamma_j}(\ell)}\right)\nonumber\\
&\phantom{=}&\phantom{\frac{\det(A_{r+s+1+j}(\ell))}{\det(A(\ell))}}=O\left(\ee^{\ell\mu_j}(-\ell)^{\frac{p\alpha_j\lambda}{\kappa}}\right)\qquad
j=1,\ldots,r-1\,.
\end{eqnarray*}
The equivalent constants $\tilde{c}_1(\ell),\ldots,\tilde{c}_r(\ell)$ and
$\tilde{b}_{-s+1}(\ell),\ldots, \tilde{b}_r(\ell)$ that belongs to the second
partial eigenfunction solve an equation system
similar to (\ref{matrixligning}):
\begin{equation}\label{matrixligningto}
\left[\begin{array}{cccccc}
0 & \ldots & 0 & N^{31}_{\Gamma_{-s+1}}(\ell) & \ldots & N^{31}_{\Gamma_{r}}(\ell)\\
\vdots  & \ddots & \vdots & \vdots & \ddots &\vdots \\
0 & \ldots & 0 & N^{3r}_{\Gamma_{-s+1}}(\ell) & \ldots &
N^{3r}_{\Gamma_{r}}(\ell)\\
M^{11}_{\Gamma_1} & \ldots & M^{11}_{\Gamma_r} & N^{11}_{\Gamma_{-s+1}} & \ldots &
N^{11}_{\Gamma_{r}}\\
\vdots  & \ddots & \vdots & \vdots & \ddots &\vdots \\
M^{1r}_{\Gamma_1} & \ldots & M^{1r}_{\Gamma_r} & N^{1r}_{\Gamma_{-s+1}} & \ldots &
N^{1r}_{\Gamma_{r}}\\
-M^{21}_{\Gamma_1} & \ldots & -M^{21}_{\Gamma_r} & N^{21}_{\Gamma_{-s+1}} & \ldots &
N^{21}_{\Gamma_{r}}\\
\vdots  & \ddots & \vdots & \vdots & \ddots &\vdots \\
-M^{2s}_{\Gamma_1} & \ldots & -M^{2s}_{\Gamma_r} & N^{2s}_{\Gamma_{-s+1}} & \ldots &
N^{2s}_{\Gamma_{r}}
\end{array}\right]\left[\begin{array}{c}
\tilde{c}_1(\ell)\\
\vdots\\
\tilde{c}_r(\ell)\\
\tilde{b}_{-s+1}(\ell)\\
\vdots\\
\tilde{b}_{r}(\ell)
\end{array}\right]=\left[\begin{array}{c}
\frac{1}{\mu_1}\\
\vdots\\
\frac{1}{\mu_r}\\
0\\
\vdots\\
0\end{array}\right]\,, 
\end{equation}
where the integration contour $\Gamma_{-s}$ is replaced by
$\Gamma_r$ in order to obtain a new and independent partial
eigenfunction. It is similarly shown that the constants have the following
asymptotics as functions of $\ell$
\begin{eqnarray*}
\tilde{c}_i(\ell)&=&O\left(\frac{1}{\mu_1}\frac{1}{N^{31}_{\Gamma_1}(\ell)}\right)=O\left(\ee^{-\ell\mu_1}(-\ell)^{\frac{p\alpha_1\lambda}{\kappa}}\right)\qquad
i=-s,\ldots,-1\\
\tilde{b}_j(\ell)&=&
O\left(\frac{1}{\mu_1}\frac{1}{N^{31}_{\Gamma_1}}(\ell)\right)=O\left(\ee^{-\ell\mu_1}(-\ell)^{\frac{p\alpha_1\lambda}{\kappa}}\right)\qquad
j=-s+1,\ldots,0\\
\tilde{b}_j(\ell)&\sim&
\left(\frac{1}{\mu_j}\frac{1}{N^{3j}_{\Gamma_j}}(\ell)\right)=O\left(\ee^{-\ell\mu_j}(-\ell)^{\frac{p\alpha_j\lambda}{\kappa}}\right)\qquad
j=1,\ldots,r\,.
\end{eqnarray*}
The asymptotic behaviour of the $f_{\Gamma_{j,2}}^2$ functions is of
interest as well. Similar to the previous analysis it is seen that for $j=-s,\ldots,-1$ is
\begin{align*}
&\lim_{\ell\to-\infty}\frac{f^2_{\Gamma_{j,2}}(\ell)}{\ee^{\ell\nu_{-j}}(-\ell)^{\frac{q\beta_j\lambda}{\kappa}-1}}=\psi_{\setminus\{-\nu_{-j}\}}(-\nu_{-j})\int_{\tilde{\Gamma}_a}z^{-\frac{q\beta_j\lambda}{\kappa}}\ee^z\;\dd
z\,,\textrm{  if $\nu_{-j}$ is a singularity}\\
&\lim_{\ell\to-\infty}\frac{f^2_{\Gamma_{j,2}}(\ell)}{\ee^{\ell\nu_{-j}}(-\ell)^{\frac{q\beta_j\lambda}{\kappa}-1}}=\psi_{\setminus\{-\nu_{-j}\}}(-\nu_{-j})\int_{\tilde{\Gamma}}z^{-\frac{q\beta_j\lambda}{\kappa}}\ee^z\;\dd
z\,,\textrm{  if $\nu_{-j}$ is a root}\,.
\end{align*}
For $j=0$ is
$$
\lim_{j\to-\infty}f^2_{\Gamma_{0,2}}(\ell)=\psi_{\setminus\{0\}}(0)\int_{\tilde{\Gamma}_a}z^{-1}\ee^z\;\dd
z\,,
$$
and for $j=1,\ldots,r$ is
\begin{eqnarray*}
\lim_{\ell\to-\infty}\frac{f^2_{\Gamma_{j,2}}(\ell)}{\ee^{-\ell\mu_{j}}(-\ell)^{\frac{p\alpha_j\lambda}{\kappa}-1}}&=&\psi_{\setminus\{\mu_{j}\}}(\mu_{j})\int_{\tilde{\Gamma}_a}z^{-\frac{p\alpha_j\lambda}{\kappa}}\ee^z\;\dd
z\,,
\quad \textrm{if $\mu_{j}$ is a singularity}\\
\lim_{\ell\to-\infty}\frac{f^2_{\Gamma_{j,2}}(\ell)}{\ee^{-\ell\mu_{j}}(-\ell)^{\frac{p\alpha_j\lambda}{\kappa}-1}}&=&\psi_{\setminus\{\mu_{j}\}}(\mu_{j})\int_{\tilde{\Gamma}}z^{-\frac{q\alpha_j\lambda}{\kappa}}\ee^z\;\dd
z\,,
\quad \textrm{if $\mu_{j}$ is a root}\,.
\end{eqnarray*}
By comparing these results with the asymptotics for the constants
$c_i(\ell)$, $\tilde{c}_i(\ell)$, $b_j(\ell)$ and $\tilde{b}_j(\ell)$ it is seen that
\begin{itemize}
\item $b_j(\ell) f^2_{\Gamma_{j,2}}(\ell)$ tends to zero exponentially fast as
  $\ell\to-\infty$ for $j=-s,\ldots,-1$
\item $\tilde{b}_j(\ell)f^2_{\Gamma_{j,2}}(\ell)$ tends to zero exponentially
  fast as $\ell\to -\infty$ for $j=-s+1,\ldots,0$
\item $b_j(\ell)f^2_{\Gamma_{j,2}}(\ell)=O\left(\frac{1}{-\ell}\right)$ for $\ell\to-\infty$ when $j=1,\ldots,r-1$
\item $\tilde{b}_j(\ell)f^2_{\Gamma_{j,2}}(\ell)=O\left(\frac{1}{-\ell}\right)$ for $\ell\to-\infty$ when $j=1,\ldots,r$\,.
\end{itemize} 
Finally, the non--zero limit of $b_0(\ell) f^{2}_{\Gamma_{0,2}}(\ell)$ when
$\ell\to -\infty$ is
\begin{eqnarray*}
\lim_{\ell\to-\infty}b_0(\ell)
f^{2}_{\Gamma_{0,2}}(\ell)&=&\lim_{\ell\to-\infty}\frac{1}{\mu_r}\frac{1}{N^{3r}_{\Gamma_0}(\ell)}f^2_{\Gamma_{0,2}}(\ell)\\
&=&\frac{1}{\mu_r}\frac{\psi_{\setminus\{0\}}(0)\int_{\tilde{\Gamma}_a}z^{-1}\ee^z\;\dd
  z}{\frac{\psi_{\setminus\{0\}}(0)}{-\mu_r}\int_{\tilde{\Gamma}_a}z^{-1}\ee^z\,\dd z}\\
&=&-1\,.
\end{eqnarray*}
Hence it has been shown that
\begin{eqnarray*}
\lim_{\ell\to-\infty}f_1(\ell)&=&
\lim_{\ell\to-\infty}\sum_{j=-s}^{r-1}b_j(\ell)f_{\Gamma_{j,2}}^2(\ell)=-1\\
\lim_{\ell\to-\infty}f_2(\ell)&=&\lim_{\ell\to-\infty}\sum_{j=-s+1}^{r}\tilde{b}_j(\ell)f_{\Gamma_{j,2}}^2(\ell)=0\,.
\end{eqnarray*}
Furthermore it is shown that all $\tilde{c}_i(\ell)$ decrease to zero
so 
$$
\lim_{\ell\to-\infty}f_2(x)=\lim_{t\to-\infty}\sum_{i=1}^r\tilde{c}_i(\ell)f^1_{\Gamma_{i,1}}(x)=0
$$
and since all $c_i$ has a non--zero limit, then $\lim_{\ell\to-\infty}f_1(x)$
is well--defined and non--zero. Therefore
\begin{align*}
&\lim_{\ell\to-\infty}\PPx(\tau<\infty)\\
&=\lim_{\ell\to-\infty}f_1(x)\frac{1-f_2(\ell)}{f_1(\ell)-f_2(\ell)}+f_2(x)\frac{f_1(\ell)-1}{f_1(\ell)-f_2(\ell)}=-\lim_{\ell\to-\infty}f_1(x)\,.
\end{align*}
\end{proof}
\noindent The asymptotic expression for $c_i(\ell)$ can found to be
$$
c_i(\ell)\sim\left((-1)^{r+s+1-i}\frac{det(M_i)}{det(M)}\frac{1}{\mu_rN^{3r}_{\Gamma_0}(\ell)}\right)\,,
$$
where 
$$
M_i=\left[\begin{array}{ccccccccc}
M^{11}_{\Gamma_1} & \ldots & M^{11}_{\Gamma_{i-1}} &
M^{11}_{\Gamma_{i+1}} & \ldots & M^{11}_{\Gamma_r} & N^{11}_{\Gamma_{-s}} & \ldots &
N^{11}_{\Gamma_{0}}\\
\vdots  & \ddots & \vdots & \vdots & \ddots & \vdots & \vdots & \ddots &\vdots \\
M^{1r}_{\Gamma_1} & \ldots & M^{1r}_{\Gamma_{i-1}} &
M^{1r}_{\Gamma_{i+1}} & \ldots & M^{1r}_{\Gamma_r} & N^{1r}_{\Gamma_{-s}} & \ldots &
N^{1r}_{\Gamma_{0}}\\
-M^{21}_{\Gamma_1} & \ldots & -M^{21}_{\Gamma_{i-1}} &
-M^{21}_{\Gamma_{i+1}} & \ldots & -M^{21}_{\Gamma_r} & N^{21}_{\Gamma_{-s}} & \ldots &
N^{21}_{\Gamma_{0}}\\
\vdots  & \ddots & \vdots& \vdots & \ddots & \vdots & \vdots & \ddots &\vdots \\
-M^{2s}_{\Gamma_1} & \ldots & -M^{2s}_{\Gamma_{i-1}} &
-M^{2s}_{\Gamma_{i+1}} &\ldots & -M^{2s}_{\Gamma_r} & N^{2s}_{\Gamma_{-s}} & \ldots &
N^{2s}_{\Gamma_{0}}
\end{array}\right]\;.
$$
Hence we have
\begin{Cor}\label{korollaret} 
For $i=1,\ldots,r$ it holds that
\begin{eqnarray*}
\lim_{\ell\to-\infty}c_i(\ell)&=&(-1)^{r+s+1-i}\frac{\det(M_i)}{\det(M)}\frac{1}{\mu_r}\left(\frac{\psi_{\setminus\{0\}}(0)}{-\mu_k}\int_{\tilde{\Gamma}_a}z^{-1}\ee^z\,\dd
  z\right)^{-1}\\
&=&
(-1)^{r+s-i}\frac{\det(M_i)}{\det(M)}\left(\psi_{\setminus\{0\}}(0)\int_{\tilde{\Gamma}_a}z^{-1}\ee^z\,\dd
  z\right)^{-1}\,.
\end{eqnarray*}
\end{Cor}
\subsection{Negative drift and the undershoot}\label{undershoot}
Consider the negative drift case,
$\kappa<0$, where the ruin probability is 1. This situation is particularly simple because only one partial
eigenfunction, $f$,
is needed, since crossing $\ell$ through continuity is not possible. The
Laplace transform of the undershoot is therefore expressed by the
simple formula
$$
\EEx[\ee^{-\zeta Z}]=f(x)\,.
$$ 
Since $\psi$ satisfies that
$|\psi(z)|=O(|z|^{-1-\tfrac{\lambda}{\kappa}})$, the negative
$\kappa$ makes infinite integration contours impossible. We shall apply Theorem~\ref{satningfin} and choose finite integration contours as described in \cite[Section 5]{art3}. However, in \cite{art3} the contours are suggested to be half--circles and circles, but that choice makes the calculations of our prblem too complicated. Thus we will use line segments instead. Note that $\mu_1$ is always a
zero for $\psi$. For each $i=2,\ldots,r$ define:\\[1mm]
If $\mu_i$ is a zero define $\Gamma_i$ as
\begin{align*}
&\{\mu_i+(-1-i)t\;:\;0\leq t\leq \frac{\mu_i-\mu_1}{2}\}\\
&\phantom{=}\cup
\{\mu_i-i(\mu_i-\mu_1)+(-1+i)t\;:\; \frac{\mu_i-\mu_1}{2}\leq t\leq
\mu_i-\mu_1\}\,.
\end{align*}
If $\mu_i$ is a singularity define $\Gamma_i$ as
\begin{align*}
 &\{\mu_i+\tfrac{a}{-\ell}+i(\mu_i+\tfrac{a}{-\ell}-\mu_1)+(1+i)t:
  -(\mu_i+\tfrac{a}{-\ell}-\mu_1)\leq t\leq-\frac{\mu_i+\frac{a}{-\ell}-\mu_1}{2}\}\\
&\phantom{=}\cup\{\mu_i+\tfrac{a}{-\ell}+(1-i)t:-\frac{\mu_i+\frac{a}{-\ell}-\mu_1}{2}\leq
t\leq 0\}\\
&\phantom{=}\cup\{\mu_i+\tfrac{a}{-\ell}+(-1-i)t:\frac{\mu_i+\frac{a}{-\ell}-\mu_1}{2}\leq
t\leq 0\}\\
&\phantom{=}\cup\{\mu_i-\tfrac{a}{-\ell}+i(\mu_i+\tfrac{a}{-\ell}-\mu_1)+(-1+i)t:
  \frac{\mu_i+\frac{a}{-\ell}-\mu_1}{2}\leq t\leq \mu_i+\tfrac{a}{-\ell}-\mu_1\}\,.
\end{align*}
A rough sketch of the two contours can be seen on Figure \ref{undershootfig}.
\begin{figure}
\setlength{\unitlength}{5mm}
\begin{center}
\begin{picture}(26,11)(-1,-6.5)
\put(-1,0){\vector(1,0){11}}
\put(0,-4){\vector(0,1){8}}
\put(1,0){\circle*{0.2}}
\put(7,0){\circle*{0.2}}
\put(1,0){\line(1,-1){3}}
\put(4,-3){\line(1,1){3}}
\put(0,-6.5){$\mu_i$ is a zero for $\psi$.}
\put(1.35,0.2){$\mu_1$}
\put(7.35,0.2){$\mu_i$}
\put(6,-2){$\Gamma_i$}
\put(12,0){\vector(1,0){11}}
\put(13,-4){\vector(0,1){8}}
\put(14,0){\circle*{0.2}}
\put(20,0){\circle*{0.2}}
\put(14,0){\line(1,1){3.5}}
\put(14,0){\line(1,-1){3.5}}
\put(17.5,-3.5){\line(1,1){3.5}}
\put(17.5,3.5){\line(1,-1){3.5}}
\put(13,-6.5){$\mu_i$ is a singularity for $\psi$.}
\put(13.3,-0.65){$\mu_1$}
\put(19.5,-0.65){$\mu_i$}
\put(20.5,0.65){$\mu_i+\tfrac{a}{-\ell}$}
\put(15,2.2){$\Gamma_i$}
\end{picture}
\caption{The choice of contours in the negative drift case\,.}
\label{undershootfig}
\end{center} 
\end{figure}
The partial eigenfunction $f$ is defined by
\begin{equation}\label{fdefinition}
f(y)=\sum_{i=2}^rc_if_{\Gamma_i}(y)+Uf^*(y)+f_0(y)\,,
\end{equation}
where $f^*(y)=1_{[\ell;\infty[}(y)$, and the parameters $c_2,\ldots, c_r$ and $U$ are the solutions of the
equation
\begin{equation}\label{matrixligningtre}
\left[\begin{array}{cccc}
-\frac{1}{\mu_1}(\ell) & M_{\Gamma_2}^1(\ell) & \cdots & M^{1}_{\Gamma_{r}}(\ell)\\
-\frac{1}{\mu_2}(\ell) & M_{\Gamma_2}^2(\ell) & \cdots & M^{2}_{\Gamma_{r}}(\ell)\\
\vdots  & \vdots & \ddots & \vdots\\
-\frac{1}{\mu_r}(\ell) & M_{\Gamma_2}^r(\ell) & \cdots & M^{r}_{\Gamma_{r}}(\ell)
\end{array}\right]\left[\begin{array}{c}
U\\
c_2\\
\vdots\\
c_r
\end{array}\right]=\left[\begin{array}{c}
-\frac{1}{\mu_1+\zeta}\\
\vdots\\
-\frac{1}{\mu_r+\zeta}
\end{array}\right] 
\end{equation}
where we shall denote the first matrix by $B(\ell)$ and the constants $M_{\Gamma_i}^k(\ell)$ are given as
\begin{equation}\label{fml:lastformula}
M_{\Gamma_i}^k(\ell)=\int_{\Gamma_i}\frac{\psi(z)}{z-\mu_k}\ee^{-\ell z}\,\dd
z
\end{equation} 
for $i=2,\ldots,r$ and $k=1,\ldots,r$. To explore the
asymptotic behaviour of $U, c_2,\ldots,c_r$ and through that the
behaviour of $f$, it is necessary to study the constants in (\ref{fml:lastformula}). 

The following result states
that the limit of the undershoot is a simple exponential distribution
with parameter $\mu_1$ from the
dominating part of the downward jumps.
\begin{Th}
For all $\zeta\geq 0$ it holds that
$$
\lim_{\ell\to-\infty}\EEx[\ee^{-\zeta Z}]= \frac{\mu_1}{\mu_1+\zeta}\,.
$$
\end{Th}
\begin{proof}
First the behaviour of the constants $M_{\Gamma_i}^k(\ell)$ when
$\ell\to-\infty$ is explored. When $\mu_i$ is a zero (for some $i=2,\ldots,r$) and
$i\neq k$ the constant can be written as
\begin{align}\label{skalberegnes}
&M_{\Gamma_i}^k(\ell)=\int_0^{\frac{\mu_i-\mu_1}{2}}(-1-i)\frac{\psi(\mu_i+(-1-i)t)}{\mu_i+(-1-i)t-\mu_k}\ee^{-\ell(\mu_i+(-1-i)t)}\,\dd
t\\
&+\int_{\frac{\mu_i-\mu_1}{2}}^{\mu_i-\mu_1}(-1+i)\frac{\psi(\mu_i-i(\mu_i-\mu_1)+(-1+i)t)}{\mu_i-i(\mu_i-\mu_1)+(-1+i)t-\mu_k}\ee^{-\ell(\mu_i-i(\mu_i-\mu_1)+(-1+i)t)}\,\dd
t\,.\nonumber
\end{align}
Rewriting the expression and applying the usual substitution $s=-\ell t$ to the first part in
(\ref{skalberegnes}) yields
\begin{align*}
M_{\Gamma_i}^k(\ell)=\ee^{-\ell\mu_i}(-\ell)^{\frac{p\lambda\alpha_i}{\kappa}-1}&\int_0^{(-\ell)\frac{\mu_i-\mu_1}{2}}(-1-i)\frac{\psi_{\setminus
    \{\mu_i\}}(\mu_i+(-1-i)\tfrac{s}{-\ell})}{\mu_i+(-1-i)\tfrac{s}{-\ell}-\mu_k}\times\\
&((-1-i)s)^{-\frac{p\lambda\alpha_i}{\kappa}}\ee^{s(-1-i)}\,\dd s\,.
\end{align*}
Hence, by dominated convergence it is seen that the integral in the
last line has the limit
\begin{eqnarray*}
&\phantom{=}&\frac{\psi_{\setminus\{\mu_i\}}(\mu_i)}{\mu_i-\mu_k}\int_0^{\infty}(-1-i)((-1-i)s)^{-\frac{p\lambda
    \alpha_i}{\kappa}}\ee^{s(-1-i)}\,\dd s\\
&=&\frac{\psi_{\setminus\{\mu_i\}}(\mu_i)}{\mu_i-\mu_k}\int_{-\Gamma}z^{-\frac{p\lambda\alpha_i}{\kappa}}\ee^{z}\,\dd z\,,
\end{eqnarray*}
where
$$
-\Gamma=\{(-1-i)t:0\leq t< \infty\}\,.
$$
Now remains to discuss the asymptotics of the second part in
(\ref{skalberegnes}). Substituting $s=-\ell(t-(\mu_i-\mu_1))$ the expression equals
\begin{eqnarray*}
&\ee^{-\ell(\frac{\mu_1+\mu_i}{2}-i\frac{\mu_i-\mu_1}{2})}(-\ell)^{-1}&\\
&\times \int_0^{(-\ell)\frac{\mu_i-\mu_1}{2}}(-1+i)\psi\left(\frac{\mu_1+\mu_i}{2}-i\frac{\mu_i-\mu_1}{2}+(-1+i)\tfrac{s}{-\ell}\right)\ee^{s(-1+i)}\,\dd
s\,.&
\end{eqnarray*}
The integral has the following limit for $\ell\to-\infty$ 
$$
\psi\left(\frac{\mu_1+\mu_i}{2}-i\frac{\mu_i-\mu_1}{2}\right)\int_{\tilde{\Gamma}}\ee^z\,\dd
z
$$
by dominated convergence, where $\tilde{\Gamma}=\{(-1+i)t:0\leq
t<\infty\}$. Since the first part grows with a larger rate than the last
part is
\begin{equation}\label{asymptotik}
\lim_{\ell\to-\infty}\frac{M_{\Gamma_i}^k(\ell)}{\ee^{-\ell\mu_i}(-\ell)^{-\frac{p\lambda\alpha_i}{\kappa}-1}}=\frac{\psi_{\setminus\{\mu_i\}}(\mu_i)}{\mu_i-\mu_k}\int_{-\Gamma}z^{-\frac{p\lambda\alpha_i}{\kappa}}\ee^{z}\,\dd
z\,.
\end{equation}
A similar result is found in the case where $i=k$:
\begin{equation}\label{asymptotikto}
\lim_{\ell\to-\infty}\frac{M_{\Gamma_i}^k(\ell)}{\ee^{-\ell\mu_i}(-\ell)^{-\frac{p\lambda\alpha_i}{\kappa}}}=\psi_{\setminus\{\mu_i\}}(\mu_i)\int_{-\Gamma}z^{-\frac{p\lambda\alpha_i}{\kappa}-1}\ee^{z}\,\dd
z\,.
\end{equation}
The same substitution technique yields results in the cases where
$\mu_i$ are singularities for $\psi$. That gives
\begin{equation}\label{asymptotiktre}
\lim_{\ell\to-\infty}\frac{M_{\Gamma_i}^k(\ell)}{\ee^{-\ell\mu_i}(-\ell)^{-\frac{p\lambda\alpha_i}{\kappa}-1}}=\frac{\psi_{\setminus\{\mu_i\}}(\mu_i)}{\mu_i-\mu_k}\int_{-\Gamma_a}z^{-\frac{p\lambda\alpha_i}{\kappa}}\ee^{z}\,\dd
z
\end{equation}
if $i\neq k$ and
\begin{equation}\label{asymptotikfire}
\lim_{\ell\to-\infty}\frac{M_{\Gamma_i}^k(\ell)}{\ee^{-\ell\mu_i}(-\ell)^{-\frac{p\lambda\alpha_i}{\kappa}}}=\psi_{\setminus\{\mu_i\}}(\mu_i)\int_{-\Gamma_a}z^{-\frac{p\lambda\alpha_i}{\kappa}-1}\ee^{z}\,\dd
z
\end{equation}
when $i=k$. Here
$$
-\Gamma_a=\{a+(1-i)t:-\infty<t\leq 0\}\cup\{a+(-1-i)t:0\leq
t<\infty\}\,.
$$
sing (\ref{asymptotik})-(\ref{asymptotikfire}) we obtain the following
asymptotic behaviour of the determinant of the matrix $B(\ell)$,
\begin{equation}\label{determinantb}
\det(B(\ell))\sim\left(-\frac{1}{\mu_1}\prod_{i=2}^rM^i_{\Gamma_i}(\ell)\right)\,.
\end{equation}
Let $B_i$ denote $B$ with the $i$th column replaced by the
vector $[-\frac{1}{\mu_1+\zeta},\ldots,-\frac{1}{\mu_r+\zeta}]^T$, then
\begin{eqnarray}
\det(B_1(\ell))&\sim&\left(-\frac{1}{\mu_1+\zeta}\prod_{i=2}^rM^i_{\Gamma_i}(\ell)\right)\label{determinantb1}\\
\det(B_i(\ell))&\sim&\left(\Big(\frac{-1}{\mu_1}\frac{1}{\mu_i+\zeta}-\frac{-1}{\mu_i}\frac{1}{\mu_1+\zeta}\Big)\prod_{j\in\{2,\ldots,r\},j=i}M^j_{\Gamma_j}(\ell)\right)\,.\label{determinantbi}
\end{eqnarray} 
The solutions of equation (\ref{matrixligningtre}) are
obtained from Cramer's rule, and the asymptotic behaviour is determined from the
results (\ref{determinantb})-(\ref{determinantbi}). This yields 
\begin{eqnarray*}
U(\ell)&=&\frac{\det(B_1(\ell))}{\det(B(\ell))}\sim\left(\frac{\frac{-1}{\mu_1+\zeta}}{\frac{-1}{\mu_1}}\right)=\frac{\mu_1}{\mu_1+\zeta}\\
c_i(\ell)&=&\frac{\det(B_i(\ell))}{\det(B(\ell))}\sim\left(\frac{\frac{-1}{\mu_1}\frac{1}{\mu_i+\zeta}-\frac{-1}{\mu_i}\frac{1}{\mu_1+\zeta}}{\frac{-1}{\mu_1}}\frac{1}{M_{\Gamma_i}^i(\ell)}\right)
\end{eqnarray*}
with $i=2,\ldots,r$. Since all $M_{\Gamma_i}^i(\ell)$ are growing
exponentially fast the asymptotics for $f$ defined in
(\ref{fdefinition}) are easily determined, as well as the limit of the
Laplace transform for the undershoot,
\begin{eqnarray*}
\lim_{\ell\to-\infty}\EEx[\ee^{-\zeta Z}]&=&
\lim_{\ell\to-\infty}\left(\sum_{i=2}^rc_i(\ell)f_{\Gamma_i}(x)+U(\ell)f^*(x)\right)\\
 &=& \lim_{\ell\to-\infty}U(\ell)\cdot 1\\
&=& \frac{\mu_1}{\mu_1+\zeta}\,.
\end{eqnarray*}
\end{proof}

\end{document}